\documentclass[letterpaper,11pt,reqno]{amsart}

\makeatletter
\usepackage{amssymb}
\usepackage{latexsym}
\usepackage{amsbsy}
\usepackage{amsfonts}
\usepackage{hyperref}
\usepackage{graphicx}

\usepackage{dsfont}

\def\marginpar#1{\ignorespaces}
\newcommand{\CC} {\mathcal{C}}

\newcommand{\MM} {\mathcal{M}}
\newcommand{\discr} {{\text{\rm finite}}}
\newcommand{\ext} {{\text{\rm ext}}}
\newcommand{\FF} {\mathcal{F}}
\newcommand{\mn} {{m \times n}}
\newcommand{\onebyn} {{1 \times n}}
\newcommand{\twobytwo} {{2 \times 2}}
\newcommand{\ZZ} {X_*}
\newcommand{\hf} {\mbox{$\frac{1}{2}$}}
\newcommand{\thf} {\mbox{$\frac{3}{2}$}}
\newcommand{\quart} {\mbox{$\frac{1}{4}$}}
\newcommand{\tquart} {\mbox{$\frac{3}{4}$}}

\newcommand{\BB} {\mathcal{B}}
\newcommand{\GG} {\mathcal{G}}
\newcommand{\HH} {\mathcal{H} }
\newcommand{\reals} {\mathbb{R} }
\newcommand{\giv}{{\,|\,}}
\newcommand{\ind}{\mathds{1}}
\newcommand{\eps}{ \varepsilon }
\newcommand{\ed}{ \stackrel{d}{=} }
\newcommand{\led}{ \le _d }
\newcommand{\convd} {\stackrel{d}{\to}}

\newtheorem{theorem}{Theorem}[section]
\newtheorem{lemma}[theorem]{Lemma}
\newtheorem{proposition}[theorem]{Proposition}
\newtheorem{corollary}[theorem]{Corollary}
\newtheorem{conjecture}[theorem]{Conjecture}
\newtheorem{example}[theorem]{Example}

\newtheorem{openpb}[theorem]{Problem}
\newtheorem{claim}[theorem]{Claim}

\numberwithin{equation}{section}
\def\feasible{coherent}
\makeatother
\begin{document}
\title[Bounds on the probability of radically different opinions]{Bounds on the probability of radically different opinions}
\author[Burdzy and Pitman]{Krzysztof Burdzy and Jim Pitman}
\address{K.B.: Department of Mathematics, University of Washington, Seattle, WA 98195} 
\email{burdzy@uw.edu}
\address{J.P.: Departments of Statistics and Mathematics, University of California, Berkeley, CA 94720} 
\email{pitman@berkeley.edu}
\thanks{Research of KB was supported in part by Simons Foundation Grant 506732. }
\begin{abstract}
We establish bounds on the probability that two different agents, who share an initial opinion 
expressed as a probability distribution on an abstract probability space, given two different sources of  information,
may come to radically different opinions regarding the conditional probability of the same event.
\end{abstract}
\maketitle

\section{Introduction}
\label{intro}

Let $A \in \FF$ be an event in some probability space $(\Omega, \FF, P )$, and let 
\begin{equation}
\label{XGYH}
X = P(A \giv \GG) \qquad \mbox{and} \qquad Y = P(A \giv \HH)
\end{equation}
for two sub-$\sigma$-fields $\GG, \HH \subseteq \FF$. Equivalently, $X$ and $Y$ are random variables
with 
\begin{equation}
\label{XYp}
0 \le X, Y \le 1 \mbox{ and } X = P(A \giv X) \mbox{ and } Y = P(A \giv Y),  \mbox{ hence } E X =  E Y = P(A) = p 
\end{equation}
for some $p \in [0,1]$.
Following \cite{Dawid1995},  
we interpret $X$ and $Y$ as the opinions of two experts about the probability of $A$ given different sources of information $\GG$ and $\HH$,
assuming the experts agree on some initial assignment of probability $P$ to events in $\FF$.
We use the term {\em \feasible{}}, 
as in \cite{Dawid1995},  
for $(X,Y)$ as in \eqref{XGYH} or \eqref{XYp}, or for the joint distribution of such $(X,Y)$ on $[0,1]^2$.
This term has been used with several other meanings in the theory of subjective probability, risk assessment, and reliability.
But we use it here only in the sense above, for two or more conditional probabilities of some common event
in a probability space. 
Note the obvious {\em reflection symmetry} that
\begin{equation}
\label{symm:ref}
\mbox{ if $(X,Y)$ is \feasible\ then so are } (Y,X), (1-X,1-Y), \mbox{ and } (1-Y,1-X).
\end{equation}
Coherent opinions $(X, Y, \ldots)$  based on 
information represented by an increasing sequence of $\sigma$-fields
form a martingale.
The notion of a \feasible{} family of random variables also includes reversed martingales, and martingales relative to a directed index set
\cite{MR553386,MR1914748}. 

As remarked in \cite[p.284]{Dawid1995}, with just a change of notation $(X,Y) \leftrightarrow (\Pi_1,\Pi_2)$, 
\begin{quote}
If $X$ and $Y$ are both produced by ``experts'', then one should not expect them to be wildly different. For example, it would seem paradoxical if, with $X$ say uniform on $[0,1]$,
one always had $Y = 1 - X$. This suggests that not all joint distributions on $[0,1]^2$ for $(X,Y)$ are \feasible.
\end{quote}
Indeed, it follows easily from Proposition \ref{prp:jk} below
that 
\begin{itemize}
\item the distribution of $(X, 1-X)$ is \feasible{} iff $P(X = \hf) = 1$; 
\item 
for non-constant \feasible{} $X$ and $Y$, the correlation $\rho(X,Y) \ne -1$.
\end{itemize}
This suggests the rough idea that \feasible{} opinions cannot be too negatively dependent. However, 
elementary examples in \cite[\S 4.1]{Dawid1995} 
show that for any prescribed value of $EX = EY = P(A) \in (0,1)$, the correlation 
between \feasible{} opinions $X$ and $Y$ about $A$ can take
any value in $(-1,1]$. 
Consider for instance, for $\delta \in (0,1)$, the distribution of $(X,Y)$
concentrated on the three points $(1-\delta, 1 - \delta)$ and $(0, 1-\delta)$ and $(1-\delta,0)$, with 
\begin{equation}
\label{deldis}
P(X=Y) = P(1-\delta,1-\delta) = \frac{ 1 - \delta}{ 1 + \delta } \qquad \mbox{and} \qquad P(0, 1-\delta)= P(1-\delta,0) = \frac{ \delta}{ 1 + \delta }  .
\end{equation}
This example from \cite{MR553386} 
gives a pair of coherent opinions $(X,Y)$ about the event $A = (X = Y)$, with
correlation $\rho(X,Y) = - \delta$ which can be any value in $(-1,0)$.

The idea expressed above, that coherent opinions $X$ and $Y$ should not be too radically different,
leads to the following precise problem, posed in \cite{KB_S} and \cite{pitman14}:
for $0 \le \delta \le 1$, evaluate
\begin{equation}
\label{xybound}
\eps(\delta) :=
\sup_{\text{\feasible{} } (X,Y)}
P( | X-Y| \ge 1 - \delta ) 
= \sup_{\text{\feasible{} } (X,Y)}
P(1- | X-Y| \le  \delta ).
\end{equation}
For $m,n = 1, 2, 3, \ldots$ consider also $\eps_{\mn}(\delta) = \eps_{n \times m}(\delta)$
defined by restricting the above supremum to $\mn$
\feasible\ $(X,Y)$, meaning that $X$ takes at most $m$ and $Y$ at most $n$ possible values.  
Let $\eps_\discr (\delta):= \sup_{ m , n } \eps_{\mn}(\delta)$, which is the supremum  in \eqref{xybound}
restricted
to $(X,Y)$ with a finite number of possible values. Each of these functions of $\delta$ is evidently 
non-decreasing and bounded above by $1$. 
Then  for all $\delta \in [0,1]$
\begin{equation}
\label{mneps}
\frac{2 \delta} { 1 + \delta} \, \le  \, \eps_{\twobytwo}(\delta) \, \le \, \eps_\discr (\delta) \, \le \, \eps(\delta)  \, \le \lim_{a \downarrow \delta} \eps_\discr(a).
\end{equation}
The first inequality is due to the example \eqref{deldis}. The second and third are obvious, and
the last is by elementary construction of $n \times n$ \feasible\ $(X_n,Y_n)$ with $|X_n - X| + |Y_n - Y| \le 2/n$ for any
coherent $(X,Y)$.
We use the notation $x \wedge y:= \min(x,y)$ and $x \vee y:= \max(x,y)$, and either $\ind_A$ or $\ind(A)$ for an indicator function whose value is $1$ if $A$ and $0$ else.

\begin{proposition}
\label{prp:bounds}
There are the following evaluations and bounds: for $\delta \in [0,1]$ and
$n \ge 2$, 
\label{prp:12}
\begin{align}
\label{eps1} \eps_{\onebyn} (\delta) &= \,\,\,\,\,\,\delta  \,\,\,\,\,\mbox{ if } \delta \in [0,\hf) \mbox{  and } 1 \mbox{ if } \delta \in [\hf , 1],\\
\label{eps22} \eps_{\twobytwo} (\delta) &= \frac{ 2 \delta }{ 1 + \delta } \mbox{ if } \delta \in [0,\hf) \mbox{  and } 1 \mbox{ if } \delta \in [\hf , 1],\\
\label{bounds} \eps_{\twobytwo} (\delta) &\le \,\, \eps(\delta) \,\, \le \,\, ( 2 \delta ) \wedge 1.
\end{align}
\end{proposition}

\begin{figure} \includegraphics[width=0.45\linewidth]{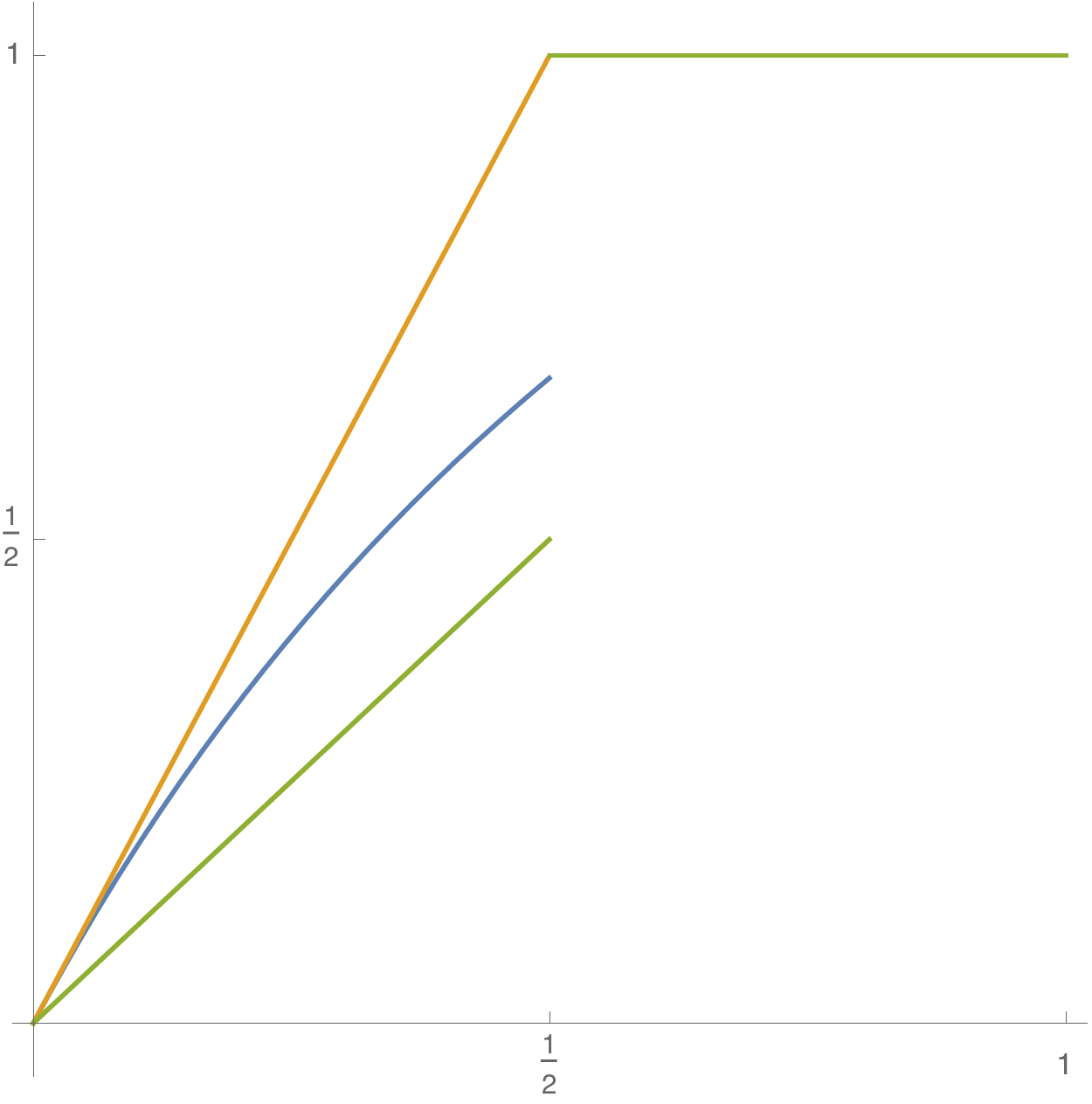}
\label{fig:fig1}
\caption{
Graphs of $\eps_{1 \times n}(\delta) \le \eps_{\twobytwo}(\delta) \le (2 \delta) \wedge 1$ in Proposition \ref{prp:bounds},
for $\delta \in [0,1]$.
}
\end{figure}

The bounds \eqref{mneps} and \eqref{bounds} were given in 
\cite[Theorem 14.1]{KB_S}, 
\cite{pitman14} 
and \cite[Theorem 18.1]{KB_R}, while \eqref{eps1} and \eqref{eps22} are new.
Our renewed interest in these results is prompted by 
\begin{claim}
{\em \cite{burdzy19}}
\label{clm:burdzy19}
$\qquad \eps_{\twobytwo} (\delta) = \eps_\discr (\delta) = \eps(\delta) \qquad$
for all $\delta \in [0,1]$.
\end{claim}
Proposition \ref{prp:12} implies that
this identity   
holds with all values $0$ for $\delta =0$, 
and all values $1$ for  $\delta \in [\hf,1]$.
The evaluations 
for $\delta \in [\hf,1]$ 
come from the \feasible\ $1\times 2$ distribution of $(X,Y)$ 
with equal probability $\hf$ at the points $(\hf,0)$ and $(\hf,1) \in [0,1]^2$.
That is
\begin{equation}
\label{bernhalf}
X = E(Y) = \hf \mbox{  for } Y = B_{1/2}
\end{equation}
where $B_p$ for $0 \le p \le 1$ denotes a random variable with the {\em Bernoulli$(p)$ distribution} 
\begin{equation}
\label{bernp}
\mbox{ $P(B_p = 1) = p$ and $P(B_p = 0) = 1-p$. }
\end{equation}
For 
$\delta \in (0,\hf)$,  
Claim \ref{clm:burdzy19} is that equality holds in all the easy 
inequalities \eqref{mneps}.  The first of these equalities
is proved here as \eqref{eps22}.  
Equality in the second inequality of \eqref{mneps} for $\delta \in (0,\hf)$ 
is much less obvious.
The proof of this in \cite{burdzy19} is at present quite long and difficult, by
recursive reduction of $m$ and $n$ for $m \times n$ coherent $(X,Y)$, until the problem is reduced to the $\twobytwo$ case
treated here by \eqref{eps22}.
We hope this exposition of the easier evaluations in Proposition \ref{prp:12}
might provoke someone to find a 
simpler proof of Claim \ref{clm:burdzy19}.

Note from \eqref{eps1} and \eqref{eps22} that each of the functions $\eps_{\onebyn}$ and $\eps_{\twobytwo}$ is continuous on each of the intervals $[0,\hf)$ and
$[\hf,1]$, but has an upward jump to $1$ at $\delta = \hf$, as shown in Figure \ref{fig:fig1}.  
If Claim \ref{clm:burdzy19}  is accepted for $\delta \in (0,\hf)$, then $\eps(\delta)$ too jumps up to $1$ at $\hf$.

Some further interpretations of these maximal probability functions,
without assuming Claim \ref{clm:burdzy19}, are presented in the following proposition,
which is a specialization of Corollary \ref{crl:compact} below.
This involves the usual notion of stochastic ordering of real random variables $V$ and $W$,
that is $V \led W$ iff $P( V > x ) \le P( W > x )$ for all real $x$. 
This is well known to be equivalent to existence of a coupling of $V$ and $W$ on a common probability space with $P( V \le W) = 1$,
and again to $E f(V) \le E f(W)$ for all bounded increasing $f$.

\begin{proposition}
\label{prp:attained}
The definition \eqref{xybound} of $\eps(\delta)$ implies that for all \feasible\ $(X,Y)$, 
\begin{equation}
\label{eq:ineq}
P(|X-Y| \ge 1- \delta) \le  \eps(\delta)  = P( \Delta \le  \delta)
\end{equation}
for a random variable $\Delta$ with $0 \le \Delta \le 1/2$. Moreover:
\begin{itemize}
\item For each $\delta \in [0,1]$ there is a \feasible\ $(X,Y)$ which attains equality in \eqref{eq:ineq}.
\item For all coherent $(X,Y)$, $|X-Y| \led 1 - \Delta$. In particular, for $r >0$
\end{itemize}
\begin{equation}
\label{rthmom}
E|X-Y|^r  = \int_0^1 r u^{r-1} P ( |X - Y | \ge u) du  \le \int_0^1 r u^{r-1} \eps(1-u) du \le   \frac{ 2 - 2^{-r}}{1 + r }   .
\end{equation}
For $\mn$ coherent $(X,Y)$, the same
conclusions hold, with the distribution of $\Delta_{\mn}$ on $[0,\hf]$ defined 
by \eqref{eq:ineq} with $\eps_{\mn}(\delta)$ in place of $\eps(\delta)$.
\end{proposition}
The  second inequality in \eqref{rthmom} uses the upper bound in \eqref{bounds}, followed by exact evaluation of the integral.
Accepting Claim \ref{clm:burdzy19} gives a slightly smaller integral involving an incomplete
beta function.  For instance, for $r=1$ these upper bounds on $E|X-Y|$ are 
$$0.75 = \tquart > \thf + \log 4 - \log 9 \approx 0.68907.$$
Corollary \ref{crl:dp80} shows that the supremum of $E|X-Y|$ over all coherent $(X,Y)$ is actually $\hf$.

The rest of this article is organized as follows.
Section \ref{sec:back} recalls some background related to Proposition \ref{prp:12},
which is proved in Section \ref{sec:bounds}.  Section \ref{sec:general} recalls some known characterizations of \feasible{}
distributions of $(X,Y)$. For reasons  we do not understand well, these general characterizations seem to be of little
help in establishing the evaluations of $\eps(\delta)$ discussed above, or in settling a number of related problems about
\feasible{} distributions, which we present in Section \ref{sec:open}. 
So much is left to be understood about the limitations on \feasible{} opinions.

\section{Background}
\label{sec:back}
\quad 
Let $(X_i, i \in I)$ 
be a finite collection of random variables defined on some common probability space $(\Omega, \FF, P )$, and suppose 
that each $X_i$ is the conditional expectation of some integrable random variable $\ZZ$ given some sub-$\sigma$-field $\FF_i$ of $\FF$:
\begin{equation}
\label{zgivx}
X_i = E( \ZZ \giv \FF_i) \qquad ( i \in I ) .
\end{equation}
Doob's well known bounds for tail probabilities and moments of the distributions of 
$\max_{i \in I }X_i $ and $\max_{i \in I }|X_i|$, for either an increasing or decreasing family of $\sigma$-fields,
and extensions of these inequalities to families of $\sigma$-fields indexed by a directed set $I$, with suitable conditional independence conditions, 
play a central role in the theory of martingale convergence. 
See for instance \cite{MR1914748,MR3449312} 
and \cite{MR3702299} 
for recent refinements of Doob's inequalities, 
and further references.
For the {\em diameter} of a martingale
\begin{equation}
\label{rangemax}
\max_{i,j \in I}  |X_i- X_j|  = \big(\max_{i \in I} X_i \big) + \big(- \max_{i \in I} (- X_i)  \big)   \le 2 \max_{i \in I} |X_i|
\end{equation}
there is no difficulty in bounding tail probabilities and moments, with an additional factor of $2$ to a suitable power. 
But finer results with best constants for the diameter have also been obtained in
\cite{MR2489169,MR3316810}. 

Much less is known about limitations on the distributions of such maximal variables  
for finite collections of $\sigma$-fields $(\FF_i, i \in I)$
without conditions of nesting or conditional independence.
We focus here on joint distributions of $X_i = E(\ZZ\giv \FF_i)$ for $\ZZ$ with $0 \le \ZZ \le 1$,
and no restrictions except $\FF_i \subseteq \FF$ in a probability  space $(\Omega, \FF, P)$.
Setting $X_J:= E[ \ZZ \giv \sigma( \cup_{i \in J} \FF_i) ]$ makes $( (X_J, \FF_J),  J \subseteq I )$ a martingale indexed by subsets of $J$ of $I$, with
$(X_i, i \in I)$ the random vector of values of this martingale on singleton subsets of $I$.
Assuming the basic probability space is sufficiently rich,
there is a random variable $U$ with uniform distribution on $[0,1]$, with $U$ independent of $\ZZ$ and $\FF_I$.
Then $\ZZ$ can be be replaced by the indicator random variable $\ind (U \le \ZZ)$.
So there is no loss of generality in supposing $\ZZ = \ind(A)$ is the indicator of some event $A$ with $P(A) = p \in [0,1]$. 
It follows that each $X_i$ is the conditional probability of $A$ given $\FF_i$:
\begin{equation}
\label{zgivxp}
X_i = P( A \giv \FF_i) \mbox{ implying } E X_i \equiv  p:= P(A)  \qquad ( i \in I ) .
\end{equation}
Then either $(X_i, i \in I)$ or its joint distribution on $[0,1]^{I}$ will be called {\em \feasible}.
Besides $EX = EY$, another necessary condition for a pair $(X,Y)$ to be \feasible{} is
provided by the following simplification and extension 
of \cite[Theorem 5.2]{Dawid1995}. See also Proposition \ref{prop:general} for some conditions that are both necessary and sufficient for $(X,Y)$ to be \feasible.

\begin{proposition}
\label{prp:jk}
Consider a pair of real-valued random variables $(X,Y)$ and assume that 
there exist disjoint intervals $G$  and $H$ and Borel sets $G' \subseteq G$ and $H' \subseteq H$ 
such that the events $(X \in G')$ and $(Y \in H')$ are almost surely identical, with $P(X \in G') >0$. 
Then there is no integrable $Z$ with $X = E(Z\giv X)$ and $Y = E(Z \giv Y)$. In particular, this condition on
$(X,Y)$ with values in $[0,1]^2$ implies that $(X,Y)$ is not coherent.
\end{proposition}
\begin{proof}
Suppose that $G' \subseteq G$ and $H' \subseteq H$. If $X = E(Z\giv X)$ and $Y = E(Z \giv Y)$ for some integrable $Z$ then 
it is easily seen that
\begin{equation}
\label{jke}
G\ni E ( Z \giv X  \in G' )= E ( Z \giv Y   \in H')\in H ,
\end{equation}
where $E ( Z \giv B)$ denotes $  E ( Z \ind_B)/P(B)$ for any $B$ with $P(B) >0$.
Since $G\cap H=\emptyset$, we obtain the conclusion.
\end{proof}

For disjoint intervals $G'=G = [0,a)$ and  $H'=H = (b,1]$,
Proposition \ref{prp:jk} yields:

\begin{corollary}
\label{crl:jk}
If $X - a$ and $Y-b$ are sure to be of opposite sign for some $0 \le a \le b \le 1$:
\begin{equation}
\label{abneg}
P( ( X - a ) (Y- b ) < 0 ) = 1,
\end{equation}
and  $P( Y > b ) > 0$,  then the distribution of $(X,Y)$ is not \feasible. 
\end{corollary}

This corrects the claim above \cite[Theorem 5.2]{Dawid1995} that 
\eqref{abneg} alone makes $(X,Y)$ not \feasible. (This is false if $P(Y > b ) = 0$; take $a = \quart, b = \tquart$ and $X = Y = \hf$).

The following construction of a \feasible\ distribution of $n$ variables $(X_1, \ldots, X_n)$ was used in
\cite{MR553386} 
to build counterexamples in the theory of almost sure convergence of martingales relative to directed sets.

\begin{example}
\label{exampledp} 
{\em(}The $(n,p)$-daisy, with $n$ petals and a Bernoulli$(p)$ center\em{) \cite{MR553386}. 
Let $A, A_1, \ldots , A_{n}$ be a measurable partition of $\Omega$ with 
$$
P(A) = p  \mbox{ and } P(A_i) = \frac{1-p}{n} \mbox{ for } 1 \le i \le n.
$$
For $1 \le i \le n$ let $\FF_i$ be the $\sigma$-field generated by $A \cup A_i$. Then set
\begin{equation} 
\label{petalx}
X_i :=  P(A \giv \FF_i) = p_n \ind ( A \cup A_i )   \mbox{ with } p_n := \frac{ n p } { n p - p + 1 } .
\end{equation}
To explain the daisy mnemonic, 
imagine $\Omega$ is the union of $n + 1$ parts of a daisy flower,
with center $A$ of area $p$, surrounded by $n$ petals $A_i$ of equal areas, with total petal area $1-p$. For each petal $A_i$, an $i$th {\em petal observer} learns  whether or not a point picked at random
from the daisy area has fallen in (the center $A$ or their petal $A_i$), or in some other petal. 
Each petal observer's conditional probability $X_i$ of $A$ is then as in \eqref{petalx}.
The sequence of $n$ variables $(X_1, \ldots, X_n)$ is both \feasible\  and exchangeable, with constant expectation $p$:
\begin{itemize}
\item given $A$ the sequence $(X_1, \ldots, X_n)$ is identically equal to  the constant $p_n$;
\item given the complement $A^c$, the sequence $(X_1, \ldots, X_n)$ is $p_n$ times an indicator sequence with a single $1$ at
a uniformly distributed index in $\{1, \ldots, n \}$.
\end{itemize}
}
\end{example}

The $(n,p)$-daisy example was designed to make $\max_{1 \le i \le n} X_i = p_n$, a constant, as large as possible with $EX_i \equiv p$.
As observed in \cite[p.224]{dp80}, this $p_n$ is the largest possible essential infimum of values of
$\max_i X_i $ 
for any \feasible\ distribution of $(X_1, \ldots, X_n)$ with $E X_i \equiv p$.
This special property involves the $n$-petal daisy in the solution in various extremal problems for \feasible\ opinions.
For instance, $(X,Y) = (X_1,X_2)$ derived from the $(2,p)$ daisy with $p = (1-\delta)/(1+ \delta)$, so $p_2 = 1-\delta$,
is the \feasible\ pair in \eqref{deldis}. This provides the lower bound for 
$\eps_{\twobytwo}(\delta)$ in \eqref{mneps},
 which according to \eqref{eps22} is attained with equality for $\delta \in [0,\hf)$.
Also:
\begin{proposition}
\label{prp:dp80}
{\em\cite{dp80}}
For every \feasible{} distribution of $(X_i, 1 \le i \le n )$  with $E X_i \equiv p$,
\begin{equation}
\label{dpineq}
E \max _{1 \le i \le n }  X_i  \le \frac{ p ( n - p ) }{ 1 + p ( n - 2 ) } .
\end{equation}
Moreover, this bound is attained by taking $(X_1, \ldots, X_{n-1})$ to be the 
$(n-1,p)$-daisy sequence, and $X_n = \ind_A$, the Bernoulli$(p)$ indicator of the  daisy center. 
\end{proposition}
For example, if
$(X_1, \ldots, X_{n})$ is the  $(n,p)$ daisy sequence, the left hand side of \eqref{dpineq} is $p_n$ in \eqref{petalx}, which
is strictly less than the right side of \eqref{dpineq}.
Proposition \ref{prp:dp80} implies:
\begin{corollary}
\label{crl:dp80}
For every \feasible\ distribution of $(X,Y)$ on $[0,1]^2$ with $E X = E Y = p$,
\begin{equation}
\label{eq:dp2}
E | X - Y | \le 2 p ( 1 - p)  \le \hf 
\end{equation}
with equality in the first inequality if $X = p$ and $Y \ed B_p$ as in \eqref{bernp}.
\end{corollary}
\begin{proof}
Take $n =2$ in \eqref{dpineq} and use $| X - Y | = 2 (X \vee Y )  - X - Y$.
\end{proof}
As noted below Proposition \ref{prp:attained}, the bound \eqref{eq:dp2} is better
than what is obtained by integration of the least upper bounds \eqref{xybound}
on tail probabilities of $|X-Y|$.
Combine \eqref{eq:dp2} with Markov's inequality to see that
\begin{equation}
\label{markov}
P( | X - Y |  \ge 1 - \delta ) \le \frac{ 2 p (1 - p ) }{( 1 - \delta)}  \le \frac{ 1 } { 2 (1 - \delta) } .
\end{equation}
But without restricting $p$ to be close to $0$ or $1$, this does not reduce the 
upper bound of \eqref{eq:ineq}. See also Problems  \ref{open:maxt} and \ref{open:fixp}.

\section{Proof of Proposition \ref{prp:bounds}}
\label{sec:bounds}

The evaluation \eqref{eps1} in Proposition \ref{prp:bounds} is implied by Lemma \ref{m17.1} for $\delta \in [0,\hf)$
and by example \eqref{bernhalf} for $\delta \in [\hf,1]$.

\begin{lemma}\label{m17.1}
If $X = E(Y \giv X )$ and $0 \le Y \le 1 $ then
$P ( | Y - X | \ge 1 - \delta ) \le \delta \mbox{ for } \delta \in [0,\hf)$,
with equality if $X = \delta$ and $Y = B_\delta$.
\end{lemma}
\begin{proof}
Suppose $X = p$ is constant and $Y = Y_p \in [0,1]$ has $E Y_p = p$. By consideration of 
$Y_{1-p} = 1 - Y_p$ it can be supposed that $p \in [0,\hf]$. But then for $\delta \in [0,\hf)$
$$
| Y_p - p | \ge 1 - \delta \mbox{ iff } Y_p \ge 1 - \delta + p ,
$$
so Markov's inequality gives
\begin{equation}
\label{yx}
P ( | Y_p - p | \ge 1 - \delta ) \le \frac{ p \, \ind ( p \le \delta ) } { 1 - \delta + p } \le \delta 
 \mbox{ for } 0 \le p \le \hf \mbox{ and } 0 \le \delta < \hf .
\end{equation}
The more general assertion of the lemma follows by conditioning on $X$.
\end{proof}

Turning to consideration of \eqref{eps22},
we start with a lemma of independent interest, which controls the variability of 
$P(A\giv G)$ as a function of $G$ with $P(G) >0$ by a bound that does not depend on $A$.
We work here with the elementary conditional probability which is the number $P(A\giv G):= P(A G)/P(G)$ rather than a random variable.
Let $G \triangle H  := GH^c \cup G^c H $ denote the symmetric difference of $G$ and $H$. 

\begin{lemma}
For events $A$, $G$ and $H$ with $P(G) > 0$ and $P(H) >0$,
\begin{equation}
\label{ghineq}
| P(A \giv  G ) - P( A \giv H ) | \le P( G \triangle H \giv G \cup H ) = 1 - \frac{ P( G H ) } { P( G ) + P( H ) -  P( GH ) } .
\end{equation}
Consequently,  for each $0 \le \delta \le 1$,
\begin{equation}
\label{ghineq1}
| P(A \giv  G ) - P( A \giv H ) | \ge 1 - \delta  \implies P( G H ) \le \frac{ \delta } { ( 1 + \delta ) } ( P( G ) + P( H ) ) .
\end{equation}
\end{lemma}
\begin{proof}
Let $p = P( G H^c), q = P( G H ), r = P( G ^c H)$ and $a = P( A \giv  G H^c), b = P( A \giv G H ), c = P( A \giv G^c H)$, with the convention that $a = 0$ if $P(G H^c) = 0$, and a similar convention for $b$ and $c$.
Then 
\begin{align}\label{f27.1}
P(A \giv  G ) - P( A \giv H )  = \frac{ p a + q b } { p + q } - \frac{q b + r c }{ q + r } \le \frac{ p + r } { p + q + r }
\end{align}
from which \eqref{ghineq}-\eqref{ghineq1} follow easily. To check the inequality in \eqref{f27.1}, observe that for fixed $p,q,r$ the difference of fractions in the middle
is obviously maximized by taking $a = 1, c = 0$. That done, the difference is a linear function of $b$, whose maximum over $0 \le b \le 1$
is attained either at $b = 0$ or at $b = 1$,
when the inequality is obvious.
\end{proof}

It is easily checked that for $p,q,r$ as above, with $p + q >0$ and $q + r >0$,
there is equality in \eqref{f27.1} iff one of the following three conditions holds, where in each case the condition on $G$, $H$, and $A$ should
be understood modulo events of probability $0$:
\begin{itemize}
\item  either $p>0, q = 0, r >0, a = 1, b = c = 0$, meaning $G \cap H = \emptyset$ and $A = G$;
\item or  $p= 0, q > 0, r >0, a = 0, b = 1, c = 0$, meaning $G \subseteq H$ and $A = G$;
\item or  $p > 0, q > 0, r = 0, a = 1, b = c = 0$, meaning $H \subseteq G$ and $A = G H^c$.
\end{itemize}
Consequently, there is equality in \eqref{ghineq} iff one of these three conditions holds, either exactly as above or with
$G$ and $H$ switched.

\begin{lemma}
\label{lmm:delta}
Suppose that $X = P(A\giv X)$ and $Y = P( A \giv Y)$ have discrete distributions. Fix $0 < \delta < 1/2$, and suppose that for each pair of possible $(x,y)$
of $(X,Y)$ with
$|y - x|  \geq 1 - \delta$ 
there is no other such pair $(x',y')$ with either $x' = x$ or $y' = y$. Then
\begin{equation}
\label{goodbound}
P( | Y  - X | \ge 1 - \delta ) \le \frac{ 2 \delta} { 1 + \delta}  \qquad (0 < \delta < 1/2).
\end{equation}
\end{lemma}
\begin{proof}
Application of \eqref{ghineq1} gives for each pair $(x,y)$ with $|y - x|  \geq 1 - \delta$ 
\begin{equation}
\label{goodbound1}
P( X = x, Y = y ) \le \frac{ \delta} { 1 + \delta } \left( P( X = x ) + P( Y = y ) \right).
\end{equation}
The assumption is that as $(x,y)$ ranges over pairs $(x,y)$ with $|y - x|  \geq 1 - \delta$, the events 
$(X= x)$ are disjoint, and so are the events $(Y = y)$. So \eqref{goodbound} follows by summation of \eqref{goodbound1} over such $(x,y)$.
\end{proof}

\begin{proof}[Proof of \eqref{eps22}]
In view of \eqref{eps1}, and the examples \eqref{bernhalf} and \eqref{deldis}, it is enough to 
establish \eqref{goodbound} for $\twobytwo$ coherent $(X,Y)$
whose possible values are contained in the $4$ corners of a rectangle $R:= [x_1, x_2] \times [y_1, y_2] \subseteq [0,1]^2$
with $x_1 < x_2$ and $ y_1 < y_2$.
Fix $0 < \delta < \hf$.  Then $\{ (x,y): | y - x| \ge 1 - \delta \} = T \cup T'$ for right triangles $T$ and $T'$ in the upper left and lower right corners of $[0,1]^2$.
If neither $T$ nor $T'$ contains two corners on the same side of $R$, then
\eqref{goodbound} holds by the above lemma.
Otherwise, by the reflection symmetries \eqref{symm:ref}, it is enough to discuss the case when $T$ contains the two
left corners 
of $R$. Then $T$ contains no more corners of $R$; for that
would make 
$$(EX,EY) \in \{ (p,p), 0 \le p \le 1 \} \cap ([0, \delta] \times [1-\delta,1]) = \emptyset \mbox{ by } \delta < \hf.$$
Finally, for  
$R$ with two left corners in $T$ and two right corners not in $T \cup T'$,
replacing $(X,Y)$ by $(X, EY)$ gives a $2 \times 1$ example with the same $P(|X - Y| \ge 1 - \delta)$, which is at most $\delta$ by \eqref{eps1}.
\end{proof}

\begin{proof}[Proof of \eqref{bounds}]
This argument from \cite{pitman14} was presented in \cite[Theorem 18.1]{KB_R},
but is included here for the reader's convenience.
The lower bound in \eqref{bounds} is obvious from \eqref{mneps}.
For the upper bound, it is enough to discuss the case $\delta \in [0,\hf)$.
Observe that 
\begin{equation}
\label{inclusion}
(|X - Y | \ge 1 - \delta) \subseteq  ( X \le \delta ,  Y \ge 1 - \delta )  \cup  ( Y \le \delta ,  X \ge 1 - \delta )  .
\end{equation}
But since $X = P(A \giv X )$ and $1-Y = P(A^c \giv Y )$,
\begin{align*}
P ( X \le \delta ,  Y \ge 1 - \delta , A ) &\le P( X \le \delta ,  A )  = E  \ind (X \le \delta) X  \le \delta P( X \le \delta), \\
P ( X \le \delta ,  Y \ge 1 - \delta , A ^c) &\le P( Y \ge 1 - \delta,  A ^c )  = E  \ind ( 1 - Y \le \delta ) (1-Y) ) \le \delta P( Y \ge 1- \delta).
\end{align*}
It follows that 
\begin{eqnarray}
\label{xdel} P( X \le \delta ,  Y \ge 1 - \delta )  &\le& \delta [ P( X \le \delta) + P( Y \ge 1 - \delta) ], \\
\label{ydel} P( Y \le \delta ,  X \ge 1 - \delta )  &\le& \delta [ P( Y \le \delta) + P( X \ge 1 - \delta) ] .
\end{eqnarray}
For $\delta < 1/2$ the events $(X \le \delta)$ and $(X \ge 1 - \delta)$ are disjoint, so
$P( X \le \delta) + P( X \ge 1 - \delta)  \le 1$, and the same for $Y$.
Add \eqref{xdel} and \eqref{ydel} and use \eqref{inclusion} to obtain the upper bound in \eqref{bounds}.
\end{proof}

\section{Coherent  distributions}
\label{sec:general}

\newcommand{\JJ}{\mathcal{J}}

The following proposition summarizes
a number of known characterizations of the set of \feasible\ distributions of $(X,Y)$,
due to \cite{dp80}, \cite{gkrs91} and \cite{Dawid1995}.  

\begin{proposition}
\label{prop:general}
Let $(X,Y)$ be a pair of random variables defined on a probability space $(\Omega, \FF, P)$,
on which there is also defined a random variable $U$ with uniform distribution, independent of $(X,Y)$.
Then the following conditions are equivalent:
\begin{itemize}
\item[(i)] The joint law of $(X,Y)$ is \feasible.
\item[(ii)] There exists a random variable $Z$ defined on $(\Omega, \FF, P)$, with $0 \le Z \le 1$, such that both
\begin{equation}
\label{zgx}
E [Z g(X)] = E [ X g(X)]\quad \mbox{ and } \quad E [Z g(Y)] = E [ Y g(Y )]
\end{equation}
either for all bounded measurable functions $g$ with domain $[0,1]$, or for all bounded continuous functions $g$.
\item[(iii)] There exists a measurable function $\phi: [0,1]^2 \mapsto [0,1]$ such that
\begin{equation}
\label{phigx}
E [\phi(X,Y) g(X)] = E [ X g(X)] \quad \mbox{ and } \quad
E [\phi(X,Y) g(Y)] = E [ Y g(Y )]
\end{equation}
either for all bounded measurable $g$, or for all bounded continuous $g$.
\item[(iv)] $E X = E Y = p$ for some $0 \le p \le 1$, and
\begin{equation}
\label{strassen}
E \big[X \ind (X \in B )\big] + E \big[ Y \ind ( Y \in C )\big] \le p + P( X \in B, Y \in C )
\end{equation}
for all $B, C \in \BB$, where $\BB$ may be either the collection of all Borel subsets of $[0,1]$,
or the collection of all finite unions of intervals contained in $[0,1]$.
\end{itemize}
\end{proposition}
\begin{proof}
Condition (i) is just (ii) for $Z$ an indicator variable, while (ii) for $0 \le Z \le 1$ implies (iii) for $\phi(X,Y) = E(Z \giv X,Y)$. Assuming (iii),
 (ii) holds with $Z = \ind (U \le \phi(X,Y))$ for the uniform $[0,1]$ variable $U$ independent of $(X,Y)$.
So (i), (ii) and (iii) are equivalent.
The equivalence of (iii) and (iv) is an instance of \cite[Theorem 6]{strassen65}, according to which for
any finite measure $m$ on $[0,1]^2$, a pair of probability distributions $Q$  and $R$ on $[0,1]$ are the marginals of the measure $\phi(x,y) m(dx \,dy) $ on $[0,1]^2$, 
for $\phi$ a product measurable function with $0 \le \phi \le 1$,  iff
$$
Q(B) + R(C ) \le 1 + m ( B \times C )
$$
for all Borel sets $B$ and $C$. This is equivalent to the same condition for all finite unions of intervals,
by elementary measure theory.
After dismissing the trivial case $p = 0$, this result is applied here to $m(\cdot) = P( (X,Y) \in \cdot ) / p $ for $X$ and $Y$ with mean $p$, with
$Q(B):= E\left[ X \ind (X \in B ) \right] /p$ and 
$R(C):= E \left[ Y \ind  ( Y \in C )\right]/p$.
\end{proof}

The characterizations (ii) and (iii) above extend easily to a coherent family $(X_i,i\in I)$,
while (iv) does not \cite[p. 288]{Dawid1995}.  

\begin{corollary}
{\em \cite{dp80}}
\label{crl:cc}
For any finite $I$, the set of \feasible\ distributions of $(X_i, i \in I)$ is a convex, compact subset of probability 
distributions on $[0,1]^I$ with the usual weak topology.
\end{corollary}
\begin{proof}

To check convexity, suppose that $(X_i, i \in I)$ is subject to the extension of \eqref{zgx}. That is  for some additional
index $* \notin I$ and $X_* = Z \in [0,1]$, 
\begin{equation}
\label{zgxx}
E [X_*  g(X_i)] = E [ X_i g(X_i)] \mbox{ for all bounded continuous } g \mbox{ and } i \in I, 
\end{equation}
and the same for $Y = (Y_i, i \in I_*)$ instead of $X$, with $I_*:= I \cup \{*\}$.
Construct these random vectors $X$ and $Y$ on a common probability space with a Bernoulli$(p)$ variable $B_p$, with $X, Y$ and $B_p$ independent. 
Let $W := B_p X + (1-B_p) Y$, so the law of $W$ is the mixture of laws of $X$ and $Y$ with weights $p$ and $1-p$.
Then \eqref{zgxx} for $X$ and $Y$ implies \eqref{zgxx}  for $W$.
The proof of sequential compactness is similar. Define 
$X$ to be a subsequential limit in distribution of some  sequence of random vectors $X_n := (X_{n,i}, i \in I _* )$
subject to \eqref{zgxx} for each $n$, to deduce \eqref{zgxx} for $X$ by bounded convergence.
\end{proof}
\begin{corollary}
\label{crl:compact}
Let $\CC$ be a non-empty set of distributions of $X = (X_i, i \in I)$ on $\reals^I$ that is compact in the topology of weak convergence,
such as coherent distributions of $X$ on $[0,1]^I$.
Let $G(x):= \sup_{\CC} P(g ( X) \le x ) $ for some particular continuous function $g$,  and $x \in \reals$,
where the $\sup_{\CC}$ is over $X$ with a distribution in $\CC$.
Then
\begin{itemize}
\item [(i)] for each fixed $x \in \reals$ there exists a distribution of  $X$  in $\CC$ with $G(x) = P(   g ( X) \le x ) $;
\item [(ii)] $G(x) = P ( \gamma \le x )$ is the cumulative distribution function of a random variable $\gamma$ which is stochastically smaller than
$g (X )$ for every distribution of $X$ in $\CC$: $\gamma \led g(X)$.
\end{itemize}
\end{corollary}

\begin{proof}
By definition of $G(x)$, for each fixed $x$ there exists a sequence of random vectors $X_n$ with distributions in $\CC$ 
such that $F_n(x):= P( g(X_n) \le x ) \uparrow G(x)$. By compactness of $\CC$, it may be supposed that 
$X_n \convd X$, meaning the distribution of $X_n$ converges to that of some $X \in \CC$.
That implies $g(X_n) \convd g(X)$. Let
$F(x):= P( g(X) \le x)$.
Since $F_n(x)$ and $F(x)$ are the probabilities assigned by the laws of $g(X_n)$ and $g(X)$ to the closed set $(-\infty,x]$, 
\cite[Theorem 29.1]{MR1324786} gives 
$$G(x) \ge F(x) \ge \limsup_n F_n(x) = G(x).$$
For (ii), the only property of a cumulative distribution function that is not an obvious property of $G$ is right continuity.
To see this, take $x_n \downarrow x$ and $X_n$ with $P( g(X_n) \le x) = F_n(x)$ such that $F_n(x_n) = G(x_n)$, and $X_n \convd X$ with distribution in $\CC$.
Let $F(x):= P( g(X) \le x)$.
Then for each fixed $m$, by the same result of \cite{MR1324786}, 
$$
F(x_m)  \ge \limsup_n F_n(x_m) \ge \limsup_n F_n(x_n) = \limsup_n G(x_n)  = G(x+).
$$
Finally, letting $m \to \infty$ gives $G(x) \ge F(x) = F(x+) \ge G(x+) \ge G(x)$.
\end{proof}

Returning to discussion of a just pair random variables $(X,Y)$ with values in $[0,1]^2$,
as in Proposition \ref{prop:general},
suppose further that $X$ and $Y$ are independent, with $E X = E Y = p$.
Then the inequality \eqref{strassen} becomes
\begin{equation}
\label{strassen2}
E X \ind (X \in B ) + E Y \ind(  Y \in C ) \le p + P( X \in B) P(Y \in C ) .
\end{equation}
It was shown in \cite[Theorem 4]{gkrs91} that this condition, just for $B = (s,1]$ and $C = (t,1]$ for $0 \le s,t \le 1$,
characterizes all possible pairs of marginal distributions on $[0,1]$ of independent $X$ and $Y$ with mean $p$
such that $(X,Y)$ is \feasible. See also \cite[Proposition 3]{MR2003175}. 

\section{Open problems}
\label{sec:open}
\begin{openpb}
\label{open:main} 
Give a simple proof of Claim \ref{clm:burdzy19}. 
\end{openpb}

A check on this claim 
is to try to confirm it first with additional assumptions, such as independence of $X$ and $Y$, using \eqref{strassen2}.
But this does not seem easy. It leads rather to:

\begin{conjecture}
\label{open:m16.1} 
If $(X,Y)$ is \feasible, and $X$ and $Y$ are independent, then
\begin{equation}
\label{m16.2}
P( | X-Y| \ge 1 - \delta ) \le  2\delta(1-\delta) \qquad\text{  for  } \delta \in [0, \hf) .
\end{equation}
\end{conjecture}

Equality is attained in \eqref{m16.2} for independent $X$ and $Y$ with 
\begin{equation}
\label{indatt}
\mbox{ $X \ed Y \ed (1-\delta) B_{1-\delta}$ and $A=(X=Y=1-\delta)$.}
\end{equation}
The method of proof of \eqref{eps22} establishes \eqref{m16.2}
for $\twobytwo$ laws of $(X,Y)$.
But like Claim \ref{clm:burdzy19}, the extension of \eqref{m16.2} to general distributions of $X$ and $Y$ seems quite challenging.

The problems solved by \eqref{eps22} for $t(X,Y) = 1 (|X - Y| \ge 1 - \delta)$ and by
the case $n=2$ of \eqref{dpineq}  for $t(X,Y) = X \vee Y$, are instances of the following more general
problem, with further variants as above, assuming $X$ and $Y$ are independent.

\begin{openpb}
\label{open:maxt} \cite[p.224]{dp80}
Given some target function $t(X,Y)$ defined on $[0,1]^2$, 
evaluate 
$\sup _\CC E t(X,Y)$, the supremum of $E t(X,Y)$ as the law of $(X,Y)$  ranges over the set $\CC$ of \feasible{} laws on $[0,1]^2$.
Or the same for $\CC(p)$, coherent laws of $(X,Y)$ with $E X = EY = p$.
\end{openpb}

This problem seems to be open even for $XY$, or $|X-Y|^r$ for $r \ne 1$, 
when \eqref{rthmom} gives only a crude upper bound.
Another instance of this problem is
to evaluate 
\begin{equation}
\label{deltap}
\eps(\delta,p):= \sup _{\CC(p)} P( | X - Y |  \ge 1 - \delta ).
\end{equation}
For each $\delta \in (0,1)$, examples of coherent $(X,Y)$ with 
\begin{equation}
\label{eqdel}
P( | X - Y |  \ge 1 - \delta ) = p(\delta) :=  2 \delta/(1 + \delta)
\end{equation}
are the $\twobytwo$ example 
\eqref{deldis}, say $(X_\delta,Y_\delta)$, 
its reflection $(1-X_\delta,1-Y_\delta)$,  
and any mixture of these two laws, 
which is a $4 \times 4$ law in $\CC(p)$ for $p$ between $p(\delta)$ and $1- p(\delta)$. So
\begin{equation}
\label{eq:between}
p(\delta) \le \eps(\delta,p) \le \eps(\delta) \mbox{ for } p \mbox{ between } p(\delta) \mbox{ and } 1- p(\delta).
\end{equation}
If Claim \ref{clm:burdzy19} is accepted, both inequalities are equalities for 
$\delta \in (0,\hf]$.  But that leaves open:
\begin{openpb}
\label{open:fixp}
Find $\eps(\delta,p)$ for $\delta \in (0,\hf]$, and $p$ not covered by \eqref{eq:between}.
\end{openpb}

For a bounded upper semicontinuous $t$, such as the indicator of a closed set, 
the $\sup_\CC E t(X,Y)$ will be attained at a distribution of $(X,Y)$ in $\ext(\CC)$, the set of extreme points of the compact, convex set $\CC$ of coherent distributions \cite{benes1991extremal}. 
This leads to:

\begin{openpb}
\label{open:ext}
\cite[p.224]{dp80}
\cite[p.273]{Dawid1995}. 
Characterize $\ext(\CC)$.
\end{openpb}

For the particular target functions $t$ involved in \eqref{dpineq} and in Claim \ref{clm:burdzy19},
the $\sup_\CC E t(X,Y)$ is attained by $2 \times 2$ distributions of $(X,Y)$. Hence the following:
\begin{conjecture}
\label{cnj:twobytwo}
Every extreme coherent law of $(X,Y)$ is a $2 \times 2$ law.
\end{conjecture}
Let $\MM$ be the convex, compact subset of $\CC$ comprising laws of two term martingales $(X,Y)$, with $X = E(Y\giv X)$, $Y \in [0,1]$.
It is elementary and well known that $\ext(\MM)$ is the set of $1 \times 2$ laws of $(p, Y_p)$ for two-valued $Y_p \in [0,1]$ with $E(Y_p) = p$. 
But the extension of this result conjectured above does not seem obvious. 
It may be relatively easy to settle whether or not every extreme $m \times n$ \feasible{} $(X,Y)$ is actually $2 \times 2$,
for some small $m$ and $n$.  
In view of \eqref{strassen},  for any particular $t$, the evaluation of $\sup _\CC E t(X,Y)$, with
restriction to a fixed set of $m$ values for $X$ and $n$ values for $Y$,
 is a linear programming problem, with a finite number of constraints depending on the given values.
This problem may be solved by modern programming techniques, at least for small $m$ and $n$. 
By solving such $2 \times 3$ problems, a solution might be found which is not attained by any $2\times 2$ coherent law.
Then Conjecture \ref{cnj:twobytwo} would be false.
On the other hand,
if Conjecture \ref{cnj:twobytwo} is true, that would increase interest in the structure of $2 \times 2$ extreme laws.
The following proposition is easily proved using \eqref{strassen}:
\begin{proposition}
\label{prp:rect}
For each a rectangle $R = [x_1,x_2] \times [y_1 ,y_2] \subseteq [0,1]^2$, let 
$\CC_{\twobytwo}(R)$ denote the set of \feasible{} laws of $(X,Y)$ on the corners of $R$. Then
\begin{itemize}
\item $\CC_{\twobytwo}(R)$ is non-empty iff $R$ intersects the diagonal $\{(p,p), 0 \le p \le 1 \}$, that is iff $x_1 \vee y_1 \le x_2 \wedge y_2$. 
\item If $x_1 \vee y_1 = x_2 \wedge y_2 = p$, then $(p,p)$ is a corner of $R$, and the unique law in $\CC_{\twobytwo}(R)$ is degenerate with $X = Y = p$.
\item If $x_1 \vee y_1 < x_2 \wedge y_2$, the set 
$\ext \, \CC_{\twobytwo}(R)$ 
of extreme points of the convex set $\CC_{\twobytwo}(R)$ is identical to the
set of all extreme coherent laws supported by the set of corners  of $R$. This set  of laws $\ext \, \CC_{\twobytwo}(R)$ 
forms a convex polygon in a $2$-dimensional affine subspace of the set of probability distributions on those corners, 
with at least $2$ and at most $8$ vertices.
\end{itemize}
\end{proposition}
Examples show that the number of vertices of this polygon varies as a function of the rectangle $R$, from $2$ if $R$ is pushed into a corner of $[0,1]^2$, to at least $6$ 
for some more central locations.  Regardless of the status of Conjecture \ref{cnj:twobytwo},
this leads to:
\begin{openpb}
\label{open:polygon}
Provide an accounting of the extreme $2 \times 2 $ \feasible\ laws of $(X,Y)$ 
which is adequate
to recover \eqref{eps22} and \eqref{eq:dp2}, and to find the extrema of $E t(X,Y)$ over $2 \times 2$ \feasible\ laws for other functions $t$,
such as $t(X,Y) = XY$ or $|X-Y|^r$ for $r >0$.
\end{openpb}

\begin{openpb}
\label{open:multi}
Extensions of above problems to $n > 2$ \feasible{} opinions. 
\end{openpb}

\section{Acknowledgments}

We are grateful to David Aldous and Soumik Pal for very helpful advice.

\bibliographystyle{alpha}
\bibliography{radical}

\end{document}